\documentclass[12pt]{amsart}
\usepackage[utf8]{inputenc}
\usepackage[T1]{fontenc}
\usepackage[top=2.5cm, bottom=2.5cm, left=2.5cm, right=2.5cm]{geometry}
\usepackage{csquotes, amsfonts, amsmath, amsthm, amssymb, stmaryrd, dsfont, enumitem, xcolor, url, graphicx, multicol, float, setspace, hyperref, tikz, multirow, esvect, colortbl, pifont, doi, url}
\usepackage[normalem]{ulem}
\usepackage[breakable, skins]{tcolorbox}
\setlist[itemize]{leftmargin=1cm}
\setlist[enumerate]{leftmargin=1cm}

\theoremstyle{plain}
\newtheorem{theo}{Theorem}[section]
\newtheorem{prop}[theo]{Proposition}

\newtheorem{coro}[theo]{Corollary}
\newtheorem{lem}[theo]{Lemma}

\theoremstyle{definition}
\newtheorem{defi}[theo]{Definition}
\newtheorem{rem}[theo]{Remark}

\hypersetup{
    colorlinks=true,
    linkcolor=black,
    urlcolor=black,
    citecolor=black
}

\newcommand{\N}{\mathbb{N}}

\newcommand{\R}{\mathbb{R}}
\newcommand{\C}{\mathbb{C}}
\newcommand{\D}{\mathbb{D}}
\newcommand{\Z}{\mathbb{Z}}
\newcommand{\T}{\mathbb{T}}
\newcommand{\Q}{\mathbb{Q}}

\newcommand{\dd}{\mathrm{d}}

\newcommand{\Hol}{\mathop{\rm Hol}\nolimits}

\newcommand{\indic}{\mathds{1}}

\newcommand{\Lcont}{\mathcal{L}}

\newcommand{\Cont}{\mathcal{C}}

\newcommand{\CC}{\mathbb{C}}
\newcommand{\DD}{\mathbb{D}}

\newcommand{\RR}{\mathbb{R}}

\newcommand{\norm}[1]{\left\Vert#1\right\Vert}

\newcommand{\abs}[1]{\left\lvert #1 \right\rvert}

\newcommand{\sbt}{\,\begin{picture}(-1,1)(-1,-3)\circle*{3}\end{picture}\ }

\newcommand{\dsps}{\displaystyle}

\renewcommand{\textbf}[1]{\begingroup\bfseries\mathversion{bold}#1\endgroup}

\author{I. Chalendar}
\address{Isabelle CHALENDAR, Université Gustave Eiffel, LAMA, (UMR 8050), 
    UPEM, UPEC, CNRS, F-77454, Marne-la-Vallée (France)}
\email{isabelle.chalendar@univ-eiffel.fr}

\author{L. Oger}
\address{Lucas OGER, Université Gustave Eiffel, LAMA, (UMR 8050), 
    UPEM, UPEC, CNRS, F-77454, Marne-la-Vallée (France)}
\email{lucas.oger@univ-eiffel.fr}

\author{J. R. Partington}
\address{Jonathan R. PARTINGTON, School of Mathematics, University of Leeds, Leeds LS2 9JT, Yorkshire, U.K.}
\email{j.r.partington@leeds.ac.uk}

\title[Linear isometries of \texorpdfstring{$\Hol(\D)$}{Hol(D)}]
{Linear isometries of \texorpdfstring{$\Hol(\D)$}{Hol(D)}}


\begin{document}

\begin{abstract}
A complete characterisation is given of all the linear isometries of the Fréchet space of all holomorphic functions on the unit disc, when it is given one of the two standard metrics: these turn out to be weighted composition operators of a particular form. 
Operators similar to an isometry are also classified. Further, the larger class of operators isometric when restricted to one of the defining seminorms is identified. Finally, the spectra of such operators are studied.
\end{abstract}

\maketitle

\section{Introduction}

Let $X$ be a topological space of holomorphic functions on the open unit disc $\D$. For $m$ a holomorphic function on $\D$ and $\varphi$ a holomorphic self-map of $\D$, then the linear operator defined by 
\[ W_{m,\varphi}(f) = m (f \circ \varphi), \quad f \in X,\]
is called the \emph{weighted composition operator} with symbols $m$ and $\varphi$. \medskip

Observe that $W_{m,\varphi}(f) = M_m C_{\varphi}(f)$, where $M_m(f) = m f$ is the \emph{multiplication operator} with symbol $m$ and $C_\varphi(f) = f \circ \varphi$ is the \emph{composition operator} with symbol $\varphi$. If $m$ is identically $1$, then $W_{m,\varphi} = C_\varphi$, and if $\varphi$ is the identity, then $W_{m,\varphi}= M_m$. Weighted composition operators are fundamental in the study of Banach and Hilbert spaces that embed continuously in $\Hol(\D)$, the algebra of all holomorphic functions on $\D$. 
Indeed, the study of the geometry of a space $X$ is centered on the identification of the linear isometries on $X$, and there is an obvious connection between weighted composition operators and isometries. \medskip

This connection can be traced back to Banach himself. In \cite{Banach}, Banach proved that the surjective isometries on ${\Cont}(K)$, the space of continuous real-valued functions on a compact metric space $K$, are of the form $T : f \mapsto m (f \circ \varphi)$, with $\abs{m} \equiv 1$ and $\varphi$ a homeomorphism of $K$ onto itself. \medskip

Although the characterisation of isometries is an open problem for most Banach spaces of holomorphic functions, there are many spaces for which the isometries are known.

\begin{itemize}[label=\sbt]
    \item On the Hardy space $H^p$ of the open unit disc $\D$, $p \neq 2$, Forelli has shown that the isometries are certain weighted composition operators (\cite{Forelli}). \medskip

    \item On the Hilbertian Hardy space $H^2$ on $\D$, there are too many isometries to have a complete characterisation. However, in \cite[Prop. 4.2]{Chal-Part}, Chalendar and Partington showed that provided that the symbol $\varphi$ is inner, there exist weighted composition
    operators
    $W_{m,\varphi}$  that are isometries. 
    Recall that on $H^2$, $C_\varphi$ is isometric if and only if $\varphi$ is inner and $\varphi(0) = 0$ (see \cite{Schwartz}), whereas $W_{m,\varphi}$ is isometric implies that $\varphi$ is inner but no extra condition on $\varphi$ is required. \medskip

    \item On the Bergman space $A^p$ of $\D$, $p\neq 2$, Kolaski (\cite{Kolaski}) showed that the surjective isometries are weighted composition operators. Zorboska (\cite{Zorboska}) initiated the study of isometric weighted composition operators on $A^2$ of $\D$. Those classes of operators were also considered in \cite{Chal-Part} on weighted Bergman spaces. \medskip

    \item On the disc algebra $A(\D)$, El-Gebeily and Wolfe (\cite{El-G-W}) showed that the isometries are of two types: either they are weighted composition operators (\guillemotleft{} Type 1 \guillemotright{} isometries), or we have to add an extension operator to these \guillemotleft{} Type 1 \guillemotright{} isometries. \medskip

    \item On the Bloch space, the first study of the isometries was made by Cima and Wogen in \cite{Cima-Wogen}. They analyzed the isometries on the subspace of the Bloch space of the open unit disc whose elements fix the origin. On this space, they showed that the surjective isometries are normalized compressions of weighted composition operators induced by disc automorphisms, whereas Colonna (\cite{Colonna}) gave a characterisation of the isometric composition operators on the Bloch space. \medskip
\end{itemize}

We will not give an exhaustive list of contributions in this area but we would like to emphasize the fact that the weighted composition operators play a central role in the study of the isometries on several spaces of holomorphic functions on $\D$. \medskip

The goal of this paper is to characterise the linear isometries on the set $\Hol(\D)$ of all holomorphic functions on the open unit disc $\D$. This is a Fréchet space, endowed with the family of seminorms defined by
\[ \norm{f}_{\infty, 1-1/p} := \sup_{\abs{z} \le 1-1/p} \abs{f(z)},
    \qquad f \in \Hol(\D), \qquad p \in \N. \]
Other sequences are possible, and the results we prove will apply to these too.
    
These seminorms are associated with the topology of uniform convergence on all compact subsets of $\D$. An isometry of $\Hol(\D)$ is simply an isometry for all the seminorms $\norm{\cdot}_{\infty, 1-1/p}$, thanks to the following lemma.

\begin{lem}\label{Lemme - Isométrie Fréchet}
    Let $X$ be a Fréchet space, and $(\norm{\cdot}_k)_{k \ge 1}$ an associated increasing family of semi-norms.
    A linear operator $T : X \to X$ is isometric, considering the distance $d_X$ defined by
    \[ d_X(x, y) = \sum_{k = 1}^{\infty} 2^{-k} \min(1, \norm{x-y}_{k}), \]
    if and only if for all $x \in X$ and $k \in \N$, $\norm{Tx}_k = \norm{x}_k$.
\end{lem}

\begin{proof}
    Assume that for all $x \in X$ and $k \in \N$, $\norm{Tx}_k = \norm{x}_k$. Then, for all $x, y \in X$,
    \begin{align*} 
        d_X(Tx, Ty) 
        & = \sum_{k = 1}^\infty \frac{1}{2^k} \min(1, \norm{Tx-Ty}_k) \\
        & = \sum_{k = 1}^\infty \frac{1}{2^k} \min(1, \norm{T(x-y)}_k)
        = \sum_{k = 1}^\infty \frac{1}{2^k} \min(1, \norm{x-y}_k)
        = d_X(x, y). 
    \end{align*}
    Hence, the operator $T$ is isometric on $X$ for the distance $d_X$. \medskip

    Conversely, if there exist $x \in X$ and $p \in \N$ such that $\norm{Tx}_p \neq \norm{x}_p$, consider the smallest such $p$  (that is $\norm{Tx}_q = \norm{x}_q$ for all $q < p$). Assume that $\norm{Tx}_p < \norm{x}_p$, and take $y = x/\norm{x}_p$.
    Recall that for $q < p < r$,
    \[ \norm{y}_q \le \norm{y}_p = 1 \le \norm{y}_r. \]
    Then, since $\norm{y}_p > \norm{Ty}_p$,
    \begin{align*} 
        d_X(y, 0) 
        = \sum_{k = 1}^\infty \frac{1}{2^k} \min(1, \norm{y}_k)
        & = \sum_{k = 1}^{p-1} \frac{1}{2^k} \norm{y}_k 
            + \frac{1}{2^p} \norm{y}_p
            + \sum_{k = p+1}^\infty \frac{1}{2^k} \\
        & > \sum_{k = 1}^{p-1} \frac{1}{2^k} \norm{Ty}_k 
            + \frac{1}{2^p} \norm{Ty}_p
            + \sum_{k = p+1}^\infty \frac{1}{2^k}\min(1, \norm{Ty}_k) \\
        & = d_X(Ty, 0).
    \end{align*}
    Finally, the operator $T$ is not isometric on $X$, for the distance $d_X$. \smallskip

    If we have the reverse inequality $\norm{Tx}_p > \norm{x}_p$, then by taking $z = x/\norm{Tx}_p$, since $\norm{z}_p < \norm{Tz}_p = 1$, the same calculations give
    \begin{align*} 
        d_X(Tz, 0) 
        = \sum_{k = 1}^\infty \frac{1}{2^k} \min(1, \norm{Tz}_k)
        & = \sum_{k = 1}^{p-1} \frac{1}{2^k} \norm{Tz}_k 
            + \frac{1}{2^p} \norm{Tz}_p
            + \sum_{k = p+1}^\infty \frac{1}{2^k} \\
        & > \sum_{k = 1}^{p-1} \frac{1}{2^k} \norm{z}_k 
            + \frac{1}{2^p} \norm{z}_p
            + \sum_{k = p+1}^\infty \frac{1}{2^k}\min(1, \norm{z}_k) \\
        & = d_X(z, 0),
    \end{align*}
    so $T$ is also not isometric on $X$ for the classical distance.
\end{proof}

A similar result holds for another commonly-used metric on $\Hol(\DD)$.
\begin{lem}
    Let $X$ be a Fréchet space, and $(\norm{\cdot}_k)_{k \ge 1}$ the associated increasing family of semi-norms.
    A linear operator $T : X \to X$ is isometric, considering the distance $d'_X$ defined by
    \[ d'_X(x, y) = \sum_{k = 1}^{\infty} 2^{-k} 
        \frac{\norm{x-y}_k}{1+\norm{x-y}_k}, \]
    if and only if for all $x \in X$ and $k \in \N$, $\norm{Tx}_k = \norm{x}_k$.
\end{lem}

\begin{proof}
    The \guillemotleft{} if \guillemotright{} implication is clear, and we now prove the \guillemotleft{} only if \guillemotright{} implication. \smallskip
    
    Let $a_k=\|x\|_k$ and $b_k=\|Tx\|_k$. 
    These are increasing sequences and we can clearly assume that 
    not all the $a_k$ are $0$ and not all the $b_k$ are 0,
    as otherwise $x=0$. So let us suppose
    that $a_k > 0$ for $k \ge k_1$ and similarly $b_k > 0$ for $k \ge k_2$ (so that $a_k=0$
    if $k <k_1$ and $b_k=0$ if $k<k_2$).
    
    We have that for each $\lambda > 0$
    \[ \sum_{k=k_1}^\infty \frac{1}{2^k} 
            \frac{\lambda a_k}{1+\lambda a_k}
        = \sum_{k=k_2}^\infty \frac{1}{2^k} 
            \frac{\lambda b_k}{1+\lambda b_k}. \]
    Divide by $\lambda$, set $\lambda=1+s$ and write $\alpha_k=1/a_k$ and $\beta_k=1/b_k$. We have
    \[ \sum_{k=k_1}^\infty \frac{1}{2^k} \frac{1}{ \alpha_k + 1+s} 
        = \sum_{k=k_2}^\infty \frac{1}{2^k} \frac{1}{ \beta_k + 1+s}, \]
    initially for $s>0$ but, because these series define $H^1(\CC_+)\cap H^2(\CC_+)$ functions on the right half-plane $\CC_+$, they define the same function. These are sums of reproducing kernels (see, for example, \cite[Sec. 1.2]{JRP}), and we may take an inner product with the function $e^{-w s} \in H^\infty(\CC_+)$ for $w >0$ to get
    \[ \sum_{k=k_1}^\infty \frac{1}{2^k} \exp(-w(\alpha_k + 1)) 
        = \sum_{k=k_2}^\infty \frac{1}{2^k} \exp(-w(\beta_k + 1)). \]
    Again, these are equal as $H^\infty(\CC_+)$ functions of $w$, so
    on $w=iy$ for $y \in \RR$ they are equal almost everywhere, and they are both almost-periodic functions of $y$, so that $\alpha_k=\beta_k$ for all $k$ (this standard fact can be found in many places: for example, \cite[Sec. 5.1]{JRP}). This completes the proof.
\end{proof}

The paper is organised as follows. First, in Section \ref{Section - Isométries Hol}, we characterise the linear isometries of $\Hol(\D)$, using only two different seminorms. We prove (Theorem \ref{Thm - 2 semi-normes}) that they are trivial weighted composition operators in the following sense: they are defined by
\[ T_{\alpha, \beta}(f)(z) = \alpha f(\beta z), \qquad
    f \in \Hol(\D), \quad z \in \D, \quad \abs{\alpha} = \abs{\beta} = 1. \]

We will mainly need the two following results. The first one comes from \cite[Proposition 2.1]{ACC} and gives a necessary and sufficient condition to be a composition operator on $\Hol(\D)$. The second one is a new characterisation of finite Blaschke products.

\begin{lem}\label{Lemme - ACC}
    Let $T : \Hol(\D) \to \Hol(\D)$ be linear and continuous. Then $T$ is a composition 
    operator if and only if for all $n \in \N_0$, $Te_n = (Te_1)^n$, with $e_n(z) = z^n$.
\end{lem}

The following is a stronger form of a result due to Kamowitz \cite[Lem.~1.3]{Kamowitz}.
Note that the proof below gives more, in that the weaker assumption that
$g$ is holomorphic in an annulus $A:=\{z \in \CC: 1-\varepsilon < |z| < 1+\varepsilon\}$ 
with $g(A \cap \overline\D) \subset \overline\D$ is sufficient 
for the conclusion on $X$. 

\begin{theo}\label{Thm - Jonathan amélioré} 
    Let $\varepsilon > 0$ and $g \in \Hol((1+\varepsilon)\D)$ such that $g(\overline \D) \subset \overline \D$.
    
    Then $X := \{\xi \in \T : \abs{g(\xi)} = 1\}$ is either finite, or is equal to $\T$. Moreover, if $X = \T$, then $g$ is a finite Blaschke product.
\end{theo}

\begin{proof}
    Note that $g(\D) \subset \D$. Indeed, by the open mapping theorem, if there exists $z_0 \in \D$ such that $\abs{g(z_0)} = 1$, then there exists some $z_1 \in \D$ close to $z_0$ such that $\abs{g(z_1)} > 1$, which is impossible. \medskip

    Consider now $g$ as a continuous map from $\T$ to $\overline \D$. Then $X = g^{-1}(\T)$ is a closed subset of $\T$ (since $g(\D) \subset \D$). 
    
    Assume that $X$ is not finite.
    Then, the set $\T \backslash X$ is an open subset of $\T$, that is a union of an infinite number of open intervals of $\T$. 
    We look at the edges of these intervals, which lie in $X$. 
    If none of these edges is a limit of points of $X$, then the set $X$ only consists of isolated points (the edges). But since $X$ is infinite, by the Bolzano--Weierstrass theorem, $X$ must contain a limit point, contradicting the assumption.
    Hence, one of the edges is indeed a limit point of $X$, denoted in the following as $z_0$. \medskip

    This point satisfies $\abs{g(z_0)} = 1$, and there exist two sequences $(u_n) \subset X$ and $(v_n) \subset \T \backslash X$ (it is sufficient to take points from the chosen interval) which tend to $z_0$, such that $\abs{g(u_n)} = 1$ and $\abs{g(v_n)} < 1$. \medskip
    
    We set $\phi_1$, $\phi_2$ two conformal maps from $\T$ to $\R \cup \{\infty\}$ defined by
    \[ \phi_1(z) = i \frac{z-z_0}{z+z_0}, \qquad
        \phi_2(z) = i \frac{z-g(z_0)}{z+g(z_0)}. \]
    They satisfy $\phi_1(z_0) = 0$ and $\phi_2(g(z_0)) = 0$. Denote $h = \phi_2 \circ g \circ \phi_1^{-1}$.
    
    Since $\phi_1^{-1}$ is holomorphic near $0$, as is $\phi_2$ near $g(z_0)$, the function $h$ is holomorphic around $0$, satisfies $h(0) = 0$ and there exist two sequences $(a_n), (b_n) \subset \R$ (corresponding to $(u_n)$ and $(v_n) \subset \T$) such that $a_n \to 0$, $b_n \to 0$, $h(a_n) \in \R$ and $h(b_n) \not\in \R$. However,
    \[ h'(0) = \lim_{n \to \infty} \frac{h(a_n)}{a_n} \in \R. \]
    
    The map $h_1 : z \mapsto h(z)/z - h'(0)$ is also holomorphic around $0$, and satisfies $h_1(a_n) \in \R$. Then, in a same way, we get $h_1'(0) = h''(0) \in \R$.
    Iterating this reasoning, all the derivatives of $h$ at $0$ must be real, so $h(b_n) \in \R$ (since $b_n \in \R$). Impossible.

    Finally, we must have $\T \backslash X = \varnothing$, i.e., $X = \T$, that is $g(\T) \subset \T$. Hence, $g$ is a finite Blaschke product, using \cite{GMR, Nikolski}.
\end{proof}

Next, in Section \ref{Section - Similarités isométries Hol}, we focus on the composition, multiplication, and weighted composition operators that are similar to a linear isometry of $\Hol(\D)$. We obtain a complete description of those operators (Theorem \ref{Thm - Similarités isométries Hol}). \smallskip

In Section \ref{Section - Isométries 1 semi-norme}, we give a characterisation of the linear isometries for a single seminorm $\norm{\cdot}_{\infty, 1-1/p}$ (Theorem \ref{Thm - 1 semi-norme}). There are more than for two seminorms, as expected, and the results from the disc algebra, by El-Gebeily and Wolfe (\cite{El-G-W}), will be in the spotlight in this section.
We also need the following lemma.

\begin{lem}\label{Lemme du train}
    Let $\varepsilon > 0$, and $f \in \Hol((1+\varepsilon)\D)$ such that $f(\overline \D) = \overline \D$ and $f((1+\varepsilon)\D) \subset (1+\varepsilon)\D$. Then, $f$ is a rotation.
\end{lem}

\begin{proof}
    Using Theorem \ref{Thm - Jonathan amélioré}, we know that $f$ is a finite Blaschke product. Moreover, the zeroes of $f$ are in $(1+\varepsilon)^{-1} \D$. We consider three cases.

    \begin{enumerate}
        \item If $f(z) = e^{i\theta} z^n$, for some integer $n \ge 2$, then for $z_0 = 1 + a \varepsilon$,
        \[ \abs{f(z_0)} = \left( 1+ a \varepsilon \right)^n > 1+\varepsilon, \]
        for $a \in (0, 1)$ sufficiently close to 1. Hence, $f((1+\varepsilon)\D) \not\subset (1+\varepsilon)\D$. \medskip

        \item Assume that $f(z) = e^{i\theta} \frac{z-\alpha}{1 - \overline \alpha z}$, for some $0 < \abs{\alpha} < \frac{1}{1+\varepsilon}$. Since $f$ is continuous, we have
        \[ f((1+\varepsilon)\D) \subset (1+\varepsilon)\D
            \; \Longrightarrow \; f((1+\varepsilon)\overline \D) \subset (1+\varepsilon)\overline \D. \]
        Hence, for all $x \in \R$, 
        \[ \abs{\frac{(1+\varepsilon)e^{ix} - \alpha}
                {1 - (1+\varepsilon)\overline \alpha e^{ix}}} \le 1+ \varepsilon
            \iff \abs{\frac{1 - \frac{\alpha e^{-ix}}{1+\varepsilon}}
                {1 - (1+\varepsilon)\overline \alpha e^{ix}}} \le 1. \]
        In particular, for $x$ such that $\alpha e^{-ix} = \abs{\alpha}$, we obtain
        \begin{equation}\label{Eq}
            1 - \frac{\abs{\alpha}}{1+\varepsilon} \le 1 - (1+\varepsilon)\abs \alpha
            \iff \frac{1}{1+\varepsilon} \ge 1+ \varepsilon.
        \end{equation} 
        This last inequality is impossible. \medskip

        \item \uline{General case}: Let $f$ be a finite Blaschke product which is not a rotation. Assume that $f$ has at least two zeroes (counting the multiplicities), and that $0$ is not the only zero of $f$. Hence, we may write
        \[ f(z) = e^{i\theta} z^n \frac{z-\alpha}{1 - \overline \alpha z}
            \times \prod_{i = 1}^k \frac{z- \alpha_i}{1 - \overline \alpha_i z}, \]
        with $0 < \abs{\alpha}, \abs{\alpha_i} < \frac{1}{1+\varepsilon}$. Once again, since $f$ is continuous, we must have
        \[ f((1+\varepsilon)\overline \D) \subset (1+\varepsilon)\overline \D. \]
        Take $z_0 = (1+\varepsilon)e^{i\arg(\alpha)}$. Then, since the inequalities (\ref{Eq}) are impossible,
        \[ \abs{\frac{z_0 -\alpha}{1 - \overline \alpha z_0}}
            = \abs{\frac{(1+\varepsilon) - \abs{\alpha}}{1 - (1+\varepsilon)\abs{\alpha}}}
            = (1+\varepsilon) \abs{\frac{1 - \frac{\abs{\alpha}}{1+\varepsilon}}
                    {1 -(1+\varepsilon) \abs{\alpha}}}
            > 1+\varepsilon. \]
        Moreover, $\abs{e^{i\theta} z_0^n} > 1$. Finally, since $\abs{z_0} > 1$ and for all $i \in \{1, \cdots, k\}$, $\abs{\alpha_i} < 1$, we get
        \[ \abs{\prod_{i = 1}^k \frac{z_0 - \alpha_i}{1 - \overline \alpha_i z_0}} > 1. \]
        Thus, $\abs{f(z_0)} > 1+\varepsilon$, a contradiction.
    \end{enumerate}

    To conclude, the function $f$ must be a rotation.
\end{proof}

Finally, in Section 5, we compute, using the results of \cite{ABBC-rot}, the spectra of the operators found in the previous sections. The common property of these weighted composition operators is that the symbol is a rotation. Hence, we have to consider separately the operators with constant weights and those whose weights are finite Blaschke products.

\section{Isometries of \texorpdfstring{$\Hol(\D)$}{Hol(D)}}
    \label{Section - Isométries Hol}

We begin by characterising all the isometric operators of $\Hol(\D)$. To do this, we first focus on the operators that are isometric for \textbf{two} different seminorms $\norm{\cdot}_{\infty, p}$.

\begin{theo}[Main theorem]\label{Thm - 2 semi-normes} \;
    
    Let $0 < r_1 < r_2 < 1$ be fixed, and $T : \Hol(\D) \to \Hol(\D)$ a linear and continuous operator such that for all $f \in \Hol(\D)$ and $i \in \{1, 2\}$,
    \[ \norm{T(f)}_{\infty, r_i} = \norm{f}_{\infty, r_i}
        = \sup_{\abs{z} \le r_i} \abs{f(z)}. \]
    Then there exist two constants $\alpha, \beta \in \T$ such that for all $f \in \Hol(\D)$,
    \[ T(f)(z) = \alpha f(\beta z) =: T_{\alpha, \beta}(f)(z). \]
\end{theo}

\begin{proof}
    We are going to use Lemma \ref{Lemme - ACC}. Recall that for all $n \in \N_0$, $e_n(z) = z^n$.
    
    \uline{Step 1}: We begin by showing that $Te_0 = \alpha e_0$, with $\abs{\alpha} = 1$.
    Indeed, since $T$ is isometric for $\norm{\cdot}_{\infty, r_1}$ and $\norm{\cdot}_{\infty, r_2}$, we obtain
    \[ \norm{Te_0}_{\infty, r_1} = \norm{Te_0}_{\infty, r_2} = 1. \]
    Using the maximum modulus principle, the map $Te_0$ is constant and unimodular. \medskip
    
    In the following, consider $\tilde T := \overline \alpha T$, so that $\tilde T$ is still isometric for the two seminorms $\norm{\cdot}_{\infty, r_1}$ and $\norm{\cdot}_{\infty, r_2}$, and $\tilde T e_0 = e_0$. \medskip

    \uline{Step 2}: Set $r = r_1$ or $r_2$. We show that for all $n \ge 1$,
    \[ r^n \T \subset (\tilde T e_n)(r\overline \D) =: K. \]

    Indeed, assume that there exists $\xi = r^n e^{i\theta}$ such that $\xi \not\in K$, which is a compact subset of $\overline \D$. Then $\delta = d(\xi, K) > 0$, and $K \subset r^n \overline \D \backslash D(\xi, \delta)$. We set $f = e_0 + e^{-i\theta} e_n$. Hence,
    \[ \tilde T f = e_0 + e^{-i\theta} \tilde T e_n. \]
    Now, compare the seminorms: we have $\norm{f}_{\infty, r} = 1 + r^n$ and
    \begin{align*} 
        \norm{\tilde Tf}_{\infty, r} 
        & = \sup_{\abs{z} \le r} \abs{1 + e^{-i\theta} (\tilde Te_n)(z)} \\
        & = \sup_{w \in K} \abs{1 + e^{-i\theta} w}
        \le \sup_{\abs{z} \le r^n, \abs{z - r^n} \ge \delta} \abs{1 + z}
            < 1 + r^n.
    \end{align*}
    This is impossible, since $\tilde T$ is isometric for $\norm{\cdot}_{\infty, r}$. Therefore, $r^n \T \subset K$. \\
    Moreover, because $\tilde T$ is an isometry, we get $K \subset r^n \overline \D$. \medskip

    In the following, we set $f_n = \tilde Te_n$, and for all $0 < r < 1$, we consider $g_{n, r}$ defined by
    \[ g_{n, r}(z) = \frac{1}{r^n} f_n(rz). \]
    Then $g_{n, r}$ is holomorphic on $r^{-1}\D$, and satisfies $\T \subset g_{n, r}(\overline \D) \subset \overline \D$. 
    Since $g_{n, r}(\D) \subset \D$, if we set $X = \{\xi \in \T : \abs{g_{n, r}(\xi)} = 1\}$, then if $X$ was finite, we would have $g_{n, r}(X)$ also finite. Impossible because $\T \subset g_{n, r}(\overline \D)$. Therefore, $X = \T$ by Theorem \ref{Thm - Jonathan amélioré}.
    This implies that the maps $B_1 = g_{n, r_1}$ and $B_2 = g_{n, r_2}$ are finite Blaschke products. Hence,
    \[ f_n(z) = \begin{cases}
        r_1^n B_1(z/r_1), & \abs{z} < r_1, \\
        r_2^n B_2(z/r_2), & \abs{z} < r_2.
    \end{cases} \]
    
    \uline{Step 3}:  Let us write all the terms of the Blaschke products, for $\abs{z} < r_1 < r_2$. There exist numbers $\alpha_1, \cdots, \alpha_{K_1}$ and $\beta_1, \cdots, \beta_{K_2}$ such that $0 < \abs{\alpha_i}, \abs{\beta_j} < 1$ and for $\abs{z} < r_1$,
    \begin{align}
        & r_1^{n-N_1} e^{i\theta_1} z^{N_1} \prod_{i = 1}^{K_1} 
            \frac{\frac{z}{r_1} - \alpha_i}{1 - \overline \alpha_i \frac{z}{r_1}}
        = r_2^{n-N_2} e^{i\theta_2} z^{N_2} \prod_{j = 1}^{K_2}  \label{Eq-1}
            \frac{\frac{z}{r_2} - \beta_j}{1 - \overline \beta_j \frac{z}{r_2}} \\
        \iff \; & r_1^{n-N_1} e^{i\theta_1} z^{N_1} \prod_{i = 1}^{K_1} 
            \frac{z - r_1 \alpha_i}{r_1 - \overline \alpha_i z}
        = r_2^{n-N_2} e^{i\theta_2} z^{N_2} \prod_{j = 1}^{K_2} 
            \frac{z - r_2 \beta_j}{r_2 - \overline \beta_j z} \label{Eq-2} %\\
        % \iff \; & r_1^n e^{i\theta_1} z^{N_1} \underbrace{\prod_{i = 1}^{K_1} 
        %     (z - r_1 \alpha_i)}_{= P_1(z)} \underbrace{\prod_{j = 1}^{K_2} (r_2 - \overline \beta_j z)}_{= P_2(z)}
        % = r_2^n e^{i\theta_2} z^{N_2} \underbrace{\prod_{j = 1}^{K_2} (z - r_2 \beta_j)}_{= P_3(z)} \underbrace{\prod_{i = 1}^{K_1} (r_1 - \overline \alpha_i z)}_{= P_4(z)}. \label{Eq-3}
    \end{align}
    
    We study the poles and zeroes of the Blaschke products. 
    In (\ref{Eq-2}), the zero 0 has same order on each side of the equation, so $N_1 = N_2 =: N$. 
    Moreover, the zeroes and poles of the two Blaschke products should coincide, so $K_1 = K_2 =: K$, and for each $i \in \{1, \cdots, K\}$, there exists $j \in \{1, \cdots, K\}$ such that
    \[ \alpha_i r_1 = \beta_j r_2 \quad \text{and} \quad
        \frac{r_1}{\overline \alpha_i} = \frac{r_2}{\overline \beta_j}. \]

    In particular, we obtain
    \[ \frac{r_1}{r_2} = \frac{\beta_j}{\alpha_i}
        = \frac{\overline \alpha_i}{\overline \beta_j}, \]
    so $\abs{\alpha_i}^2 = \abs{\beta_j}^2$, and $r_1 = r_2$. Impossible.
    Ergo, the Blaschke products $B_1$ and $B_2$ are of the form
    \[ B_1(z) = e^{i\theta_1} z^{N_1}, \qquad B_2(z) = e^{i\theta_2} z^{N_2}. \]

    \uline{Step 4}:  Putting the formulas in the expression of $f_n$, we get
    \[ f_n(z) = r_1^n e^{i\theta_1} \frac{z^{N_1}}{r_1^{N_1}}
        = r_2^n e^{i\theta_2} \frac{z^{N_2}}{r_2^{N_2}}, \qquad 
                \abs{z} < r_1 \le r_2. \]
                
    The order of the zero $0$ is unique, so we obtain $N_1 = N_2 =: N$, and
    \[ r_1^{n-N} e^{i\theta_1} z^N = r_2^{n-N} e^{i\theta_2} z^N. \]
    
    Identifying the modulus and argument of the coefficients, we have $r_1^{n-N} = r_2^{n-N}$, so $n=N$ (since $r_1 \neq r_2$), and $e^{i\theta_1} = e^{i\theta_2} =: \delta_n$, that is
    \[ f_n(z) = e^{i\theta} z^n = \delta_n e_n(z), \qquad \abs{z} < r_1. \]

    Using analytic continuation, the last equality is valid for $z \in \D$, so we have proved that for all $n \in \N$, there exists $\delta_n \in \T$ such that $\tilde Te_n = \delta_n e_n$. \medskip

    \uline{Step 5}:  Since $\tilde T$ is isometric, for all $r \in \{r_1, r_2\}$,
    \begin{align*} 
        \sup_{\abs{z} = r} \abs{\tilde T(e_k + e_{k+1} + e_{k+2})(z)}
        & = \sup_{\abs{z} = r} \abs{(\delta_k e_k 
            + \delta_{k+1}e_{k+1} + \delta_{k+2}e_{k+2})(z)} \\
        & = \sup_{\abs{z} = r} \abs{\delta_k z^k 
            + \delta_{k+1}z^{k+1} + \delta_{k+2} z^{k+2}} \\
        & = \sup_{\abs{z} = r} \abs{z^k + z^{k+1} + z^{k+2}}.
    \end{align*}

    Dividing by $z^k$, we obtain
    \[ \sup_{\abs{z} = r} \abs{\delta_k + \delta_{k+1}z + \delta_{k+2}z^2}
        = \sup_{\abs{z} = r} \abs{1 + z + z^2} 
        = 1 + r + r^2. \]

    Factorising by $\delta_k$, since $\abs{\delta_k} = 1$, we have
    \[ \sup_{\abs{z} = r} \abs{1 + \frac{\delta_{k+1}}{\delta_k} z 
            + \frac{\delta_{k+2}}{\delta_k} z^2}
        = \sup_{\abs{z} = r} \abs{1 + z + z^2} 
        = 1 + r + r^2. \]

    Using the triangle inequality, for all $z$ satisfying $\abs{z} = r$,
    \[ \abs{1+\frac{\delta_{k+1}}{\delta_k}z+\frac{\delta_{k+2}}{\delta_k}z^2}
        \le 1 + r + r^2, \]
    with equality if and only if
    \[ 0 = \arg(1) \equiv \arg\left[\frac{\delta_{k+1}}{\delta_k}z\right] 
        \equiv \arg\left[\frac{\delta_{k+2}}{\delta_k}z^2\right] \mod 2\pi. \]
        
    Denoting $\theta = \arg[z]$, this gives
    \[ 0 \equiv \arg[\delta_{k+1}] - \arg[\delta_k] + \theta 
        \equiv \arg[\delta_{k+2}] - \arg[\delta_k] + 2\theta \mod 2\pi. \]
        
    Finally,
    % \begin{equation}\label{Eq - Angles de rotation}
    %     \begin{cases}
    %         2\theta + 2\arg[\delta_{k+1}] - 2\arg[\delta_k] 
    %             & \equiv 0 \mod 4\pi \\
    %         2\theta + \arg[\delta_{k+2}] - \arg[\delta_k] 
    %             & \equiv 0 \mod 2\pi,
    %     \end{cases}
    % \end{equation}
    %$\arg[\delta_k] - 2\arg[\delta_{k+1}] + \arg[\delta_{k+2}] \equiv 0 \mod 2\pi$. \medskip
    \begin{align*}
        & \arg{[\delta_k]} - 2\arg{[\delta_{k+1}]} + \arg{[\delta_{k+2}]} \\
        = \; & (\arg{[\delta_{k+2}]} - \arg{[\delta_k]} + 2\theta)
            - 2(\arg{[\delta_{k+1}]} - \arg{[\delta_k]} + \theta)
        \equiv 0 \mod 2\pi.
    \end{align*}
    
    \uline{Step 6}:  We show by induction that for all $n \in \N_0$,
    $\arg[\delta_n] \equiv n\arg[\delta_1] \mod 2\pi$.
    \begin{itemize}[label=\sbt]
        \item $n = 0$ or $1$: It is immediate (see Step 1 for $n = 0$). \medskip

        \item Assume that the formula is valid for some nonnegative integers $n$ and $n+1$. From Step 5, $\arg{[\delta_n]} - 2\arg{[\delta_{n+1}]} + \arg{[\delta_{n+2}]} \equiv 0 \mod 2\pi$.
        Using induction hypothesis, we have 
        \[ n\arg[\delta_1] - 2(n+1)\arg[\delta_1] + \arg[\delta_{n+2}] \equiv 0 \mod 2\pi. \]
        After rearranging the terms, we obtain $\arg[\delta_{n+2}] \equiv (n+2)\arg[\delta_1] \mod 2\pi$.

        % \item Assume that the formula is valid for $n \ge 2$. Using (\ref{Eq - Angles de rotation}),
        % \[ \arg[\delta_{n-1}] - 2\arg[\delta_n] + \arg[\delta_{n+1}] 
        %     \equiv 0 \mod 2\pi. \]
        
    \end{itemize}

    \fbox{Conclusion:}  For all $n \in \N_0$, $\tilde T(e_n) = \delta_1^n e_n$.
    Hence, $\tilde T(e_n) = (\tilde T(e_1))^n$, so $\tilde T$ is a composition operator, with symbol $\varphi = \tilde T(e_1) = \delta_1 e_1$, by Lemma \ref{Lemme - ACC}.
    Finally, multiplying by $\alpha$, for all $f \in \Hol(\D)$,
    \[ (Tf)(z) = \alpha f(\delta_1 z). \qedhere \]
\end{proof}

We obtain the following corollary.

\begin{coro}
    The only linear operators $T : \Hol(\D) \to \Hol(\D)$ that are isometries of $\Hol(\D)$ are those of the form $T_{\alpha, \beta}$, with $\abs{\alpha} = \abs{\beta} = 1$.
\end{coro}

\begin{proof}
    Let $T$ be a linear isometry of $\Hol(\D)$. Then, in particular, $T$ is a linear isometry for two different seminorms associated with $\Hol(\D)$, using Lemma \ref{Lemme - Isométrie Fréchet}. Theorem \ref{Thm - 2 semi-normes} concludes. \medskip

    Conversely, let $\abs{\alpha} = \abs{\beta} = 1$. Then, for all $p \in \N$ and $f \in \Hol(\D)$,
    \[ \norm{T_{\alpha, \beta}(f)}_{\infty, 1-1/p}
        = \sup_{\abs{z} \le 1-1/p} \abs{\alpha f(\beta z)}
        = \sup_{\abs{w} \le 1-1/p} \abs{f(w)} = \norm{f}_{\infty, 1-1/p}. \]
    Hence, $T_{\alpha, \beta}$ is a linear isometry of $\Hol(\D)$.
\end{proof}

\section{Operators similar to an isometry on \texorpdfstring{$\Hol(\D)$}{Hol}}
    \label{Section - Similarités isométries Hol}

In this section, we will focus on linear operators of $\Hol(\D)$ that are \textit{similar} to an isometry of $\Hol(\D)$. Let us recall the concept of \textit{similarity}.

\begin{defi}\label{Def - Similaire}
    Two linear operators $T, V \in \Lcont(\Hol(\D))$ are \textbf{similar} if there exists some $U \in \Lcont(\Hol(\D))$ invertible such that $U^{-1}TU = V$.
\end{defi}

In the following, we will only consider some classes of operators: composition, multiplication, and weighted composition ones. The goal of this section is to obtain a characterisation of such operators that are similar to a linear isometry of $\Hol(\D)$. The main theorem is the following.

\begin{theo}\label{Thm - Similarités isométries Hol}
    Let $T$ be a continuous operator on $\Hol(\D)$.
    \begin{enumerate}
        \item $T = C_\varphi$ is a composition operator similar to a linear isometry of $\Hol(\D)$ if and only if $\varphi$ is bijective and elliptic, i.e., with a fixed point on $\D$. \medskip

        \item $T = M_m$ is a multiplication operator similar to a linear isometry of $\Hol(\D)$ if and only if $m$ is constant and unimodular. \medskip

        \item $T = W_{m, \varphi}$ is a weighted composition operator similar to a linear isometry of $\Hol(\D)$ if and only if $\varphi$ is bijective and elliptic with $\sigma_p(W_{m, \varphi}) \cap \T \neq \varnothing$.
    \end{enumerate}
\end{theo}

It is now time to prove this result. We recall that the linear isometries of $\Hol(\D)$ are defined by
\[ T_{\alpha, \beta}(f)(z) = \alpha f(\beta z), \qquad
    f \in \Hol(\D), \qquad \abs{\alpha} = \abs{\beta} = 1. \]
as we proved in Theorem \ref{Thm - 2 semi-normes}.

%%%%%%%%%%%%%%%%%%%%%%%%%%%%%%%%%%%%%%%%%%%%%%%%%%%%%%%%%%%%%%%%%%%%%%%%%%
\subsection{Composition operators - Proof of Theorem \ref{Thm - Similarités isométries Hol}.(\textit{i})}\label{Sous-section - Opérateurs de composition} \; \medskip

First, assume that $\varphi$ is bijective and elliptic. Denote by $\alpha \in \D$ its fixed point. Consider the map $\psi$ defined by
\begin{equation*}\label{Definition de psi}
    \psi(z) = \frac{\alpha-z}{1-\overline \alpha z}.
\end{equation*}

Then, $\psi \in \Hol(\D)$, $\psi$ is bijective, $\psi^{-1} = \psi$, and $\psi(\alpha) = 0$. Hence, $\tilde \varphi = \psi \circ \varphi \circ \psi$ is a bijective self-map of $\D$, with a fixed point at $0$. By Schwarz's Lemma (\cite{Abate, Rudin}), there exists $\beta \in \C$ such that $\abs{\beta} = 1$ and $\tilde \varphi(z) = \beta z$. Hence,
\begin{equation*}
    C_\psi \circ C_\varphi \circ C_\psi = C_{\tilde \varphi} = T_{1, \beta}.
\end{equation*}
The operator $C_\varphi$ is therefore similar to a linear isometry of $\Hol(\D)$. \medskip

Now, consider $\varphi$ not both bijective and elliptic. If $C_\varphi$ was similar to a linear isometry of $\Hol(\D)$, then we would write
\[ C_\varphi = UVU^{-1} \iff U^{-1} C_\varphi U = V, \]
with $V$ isometric. Iterating this equality $n$ times, and we get
\[ U^{-1} C_\varphi^n U = U^{-1} C_{\varphi^{[n]}} U = V^n. \]

Thanks to Lemma \ref{Lemme - Isométrie Fréchet}, for all $n, p \in \N$ and $f \in \Hol(\D)$,
\[ \norm{(U^{-1} C_\varphi^n U)(f)}_{\infty, 1-1/p}
    = \norm{V^n(f)}_{\infty, p} = \norm{f}_{\infty, 1-1/p}. \]
    
Since $U$ is invertible, writing $g = U(f)$, for all $p \in \N$ and $g \in \Hol(\D)$,
\[ \norm{(U^{-1}C_\varphi^n)(g)}_{\infty,1-1/p}
    = \norm{U^{-1}(g \circ \varphi^{[n]})}_{\infty, 1-1/p}
    = \norm{U^{-1}(g)}_{\infty, 1-1/p}. \]
    
By the Denjoy--Wolff theory (\cite{BCD, Cowen-MacCluer}), there exists a point $\omega \in \overline{\D}$ such that the iterates $\varphi^{[n]}$ of $\varphi$ converge uniformly on all compact subsets of $\D$ to $\omega$. Finally, considering $g(z) = z - \omega$, we have $g \in \Hol(\D)$, and using the continuity of $U$,
\[ \norm{U^{-1}(g)}_{\infty, 1-1/p}
    = \norm{U^{-1}(\varphi^{[n]} - \omega)}_{\infty, 1-1/p} 
        \xrightarrow[n \to +\infty]{} 0. \]
        
However, $g \not\equiv 0$, so $U^{-1}(g) \not\equiv 0$ (because $U$ is invertible). Thus, there exists $p \in \N$ such that $\norm{U^{-1}(g)}_{\infty, 1-1/p} \neq 0$, leading to an absurdity. \medskip

\fbox{Conclusion:} $C_\varphi$ is similar to a linear isometry of $\Hol(\D)$ if and only if $\varphi$ is bijective and elliptic.
The proof of Theorem \ref{Thm - Similarités isométries Hol}.$(i)$ is then complete.

%%%%%%%%%%%%%%%%%%%%%%%%%%%%%%%%%%%%%%%%%%%%%%%%%%%%%%%%%%%%%%%%%%%%%%%%%%
\subsection{Multiplication operators - Proof of Theorem \ref{Thm - Similarités isométries Hol}.(\textit{ii})}\label{Sous-section - Opérateurs de multiplication} \; \medskip

Assume that $m$ is constant and unimodular. Denote by $\alpha$ this constant. Therefore, for all $f \in \Hol(\D)$, we may write
\[ [M_m(f)](z) = m(z)f(z) = \alpha f(z) = [T_{\alpha, 1}(f)](z). \]
Hence, $M_m = T_{\alpha, 1}$ is a linear isometry of $\Hol(\D)$. \medskip

Conversely, let $M_m$ be similar to a linear isometry of $\Hol(\D)$. We may write $M_m = UVU^{-1}$,
with $U$ invertible and $V$ isometric. Iterating this $n$ times, we obtain
\[ M_m^n = M_{m^n} = UV^n U^{-1}, \]
with $V^n$ also isometric (by induction). We will consider two cases.

\begin{itemize}[label=\sbt]
    \item Assume that for all $z \in \D$, $\abs{m(z)} < 1$. Then, for all $p \in \N$, $\norm{m}_{\infty, 1-1/p} = c < 1$. Hence,
    \begin{equation}\label{Eq - Normes de m^n}
        \norm{m^n}_{\infty, 1-1/p} = c^n \xrightarrow[n \to \infty]{} 0.
    \end{equation}
    Consider $\indic$, the map defined by $\indic(z) = 1$ for all $z \in \D$. Then, for all $n \in \N_0$, since $V^n$ is isometric and $V^n U^{-1} = U^{-1} M_{m^n}$, for all $p \in \N$, by Lemma \ref{Lemme - Isométrie Fréchet},
    \begin{align*}
        \norm{U^{-1}(\indic)}_{\infty, 1-1/p}
        = \norm{V^nU^{-1}(\indic)}_{\infty, 1-1/p}
        & = \norm{U^{-1}M_{m^n}(\indic)}_{\infty, 1-1/p} \\
        & = \norm{U^{-1}(m^n)}_{\infty, 1-1/p} 
        \xrightarrow[n \to +\infty]{} 0,
    \end{align*} 
    using the continuity of $U$ and (\ref{Eq - Normes de m^n}).
    But $U^{-1}(\indic) \not\equiv 0$ (since $U$ is invertible), so there exists $p \in \N$ such that $\norm{U^{-1}(\indic)}_{\infty, 1-1/p} \neq 0$.
    Impossible. \medskip

    \item Assume that there exists $z_0 \in \D$ such that $\abs{m(z_0)} > 1$.
    Then, for all $p \in \N$, $n \in \N_0$ and $f \in \Hol(\D)$, we have
    \[ \norm{m^n f}_{\infty, 1-1/p} 
        = \norm{UV^n U^{-1}(f)}_{\infty, 1-1/p}. \]
    In particular, for $p = p_0$ such that $\abs{z_0} \le 1 - \frac{1}{p_0}$ and $f = \indic$,
    \[ \norm{U V^n U^{-1}(\indic)}_{\infty, 1-1/p_0} 
        = \norm{m^n}_{\infty, 1-1/p_0} 
        \xrightarrow[n \to \infty]{} +\infty. \]
    However, $U$ is continuous and $V$ is isometric, so there exist $C > 0$ and $q \in \N$ such that
    \[ \norm{U V^n U^{-1}(\indic)}_{\infty, 1-1/p_0} 
        \le C \norm{U^{-1}(\indic)}_{\infty, 1-1/q} < +\infty. \medskip \]
\end{itemize}

The two cases lead to an absurdity. Ergo, for all $z \in \D$, $\abs{m(z)} \le 1$, and there exists $z_0 \in \D$ such that $\abs{m(z_0)} = 1$. Using the maximum modulus principle, we have shown that $m$ is constant, of modulus one. \medskip

\fbox{Conclusion:} $M_m$ is similar to a linear isometry of $\Hol(\D)$ if and only if $m$ is constant and unimodular.
The proof of Theorem \ref{Thm - Similarités isométries Hol}.$(ii)$ is then complete.

\begin{rem}
    We have proved that there is no multiplication operator similar to a linear isometry of $\Hol(\D)$, other than the isometries themselves.
\end{rem}

%%%%%%%%%%%%%%%%%%%%%%%%%%%%%%%%%%%%%%%%%%%%%%%%%%%%%%%%%%%%%%%%%%%%%%%%%%
\subsection{Weighted composition operators - Proof of Theorem \ref{Thm - Similarités isométries Hol}.(\textit{iii})}\label{Sous-section - Opérateurs de composition pondérés} \; \medskip

To obtain the third and last result of Theorem \ref{Thm - Similarités isométries Hol}, we begin with a simple observation.

\begin{lem}\label{Lemme - Prop similaire isométries comp pond}
    Let $\varphi : \D \to \D$ holomorphic, and $m \in \Hol(\D)$.
    
    If $W_{m, \varphi}$ is similar to a linear isometry of $\Hol(\D)$, then $W_{m, \varphi}$ is invertible.
\end{lem}

\begin{proof}
    If the operator $W_{m, \varphi}$ is similar to a linear isometry of $\Hol(\D)$, then we have
    \[ W_{m, \varphi} = UVU^{-1}, \]
    with $U$ invertible and $V$ isometric. Since all linear isometries of $\Hol(\D)$ are invertible (Theorem \ref{Thm - 2 semi-normes} and \cite[Proposition 2.1]{ABCC}), $W_{m, \varphi}$ is also invertible, with inverse $UV^{-1}U^{-1}$.
\end{proof}

Therefore, in the following, using \cite[Proposition 2.1]{ABCC}, we may consider only weighted composition operators with symbol $\varphi$ bijective, and weight $m$ such that $\forall z \in \D, m(z) \neq 0$. \medskip

\uline{1\textsuperscript{st} case}: Assume that $\varphi$ is not elliptic. Then, by \cite[Proposition 2.3, Theorems 6.1 \& 7.1]{ABCC}, there are only two possiblities: either $\sigma_p(W_{m, \varphi}) = \C^*$, or $\sigma_p(W_{m, \varphi}) = \varnothing$.

If $\sigma_p(W_{m, \varphi}) = \C^*$, consider $f \in \Hol(\D) \backslash \{0\}$ such that $W_{m, \varphi}(f) = 2f$.
Suppose that $W_{m, \varphi}$ is similar to a linear isometry of $\Hol(\D)$. We may write $W_{m, \varphi} = UVU^{-1}$, with $U$ invertible and $V$ isometric. An easy induction gives $W_{m, \varphi}^n = UV^n U^{-1}$ for $n \in \N$, so
\[ W_{m, \varphi}^n (f) = 2^n f = UV^n U^{-1}f. \]

Since $U$, $U^{-1}$ and $V$ are continuous, there exists a constant $c > 0$ such that
\[ \norm{UV^n U^{-1}f}_{\infty, 1/2} \le c. \]

But $f \not\equiv 0$, so $\norm{f}_{\infty, 1/2} > 0$, and
\[ \norm{W_{m, \varphi}^n (f)}_{\infty, 1/2} 
    = \norm{2^n f}_{\infty, 1/2} = 2^n \norm{f}_{\infty, 1/2}
    \xrightarrow[n \to \infty]{} +\infty. \]
We have obtained an absurdity. \medskip

If $\sigma_p(W_{m, \varphi}) = \varnothing$ and $W_{m, \varphi}$ is similar to a linear isometry $V$ of $\Hol(\D)$, then
\[ \varnothing = \sigma_p(W_{m, \varphi})
    = \sigma_p(UVU^{-1}) = \sigma_p(V), \]
since $U$ is invertible. But $\sigma_p(V) \neq \varnothing$, using Theorem \ref{Thm - 2 semi-normes} and \cite[Theorems 3.4 \& 3.7]{ABBC-rot}.
\medskip

\fbox{1\textsuperscript{st} conclusion:} In order to make $W_{m, \varphi}$ similar to a linear isometry of $\Hol(\D)$, the symbol $\varphi$ must be bijective and elliptic. \medskip

Now, we consider symbols $\varphi$ that are elliptic and bijective. Without loss of generality (using the map $\psi$ p.\pageref{Definition de psi}), we will assume that the fixed point of $\varphi$ is $0$.
Using Schwarz's Lemma (\cite{Abate, Rudin}), we can write
\[ \varphi(z) = \beta z, \qquad \abs{\beta} = 1. \]

\begin{defi}\label{Def - Périodique et apériodique}
    Let $\varphi : z \mapsto \beta z$, with $\abs{\beta} = 1$.
    \begin{itemize}[label=\sbt]
        \item The map $\varphi$ is \textbf{periodic} if there exists $N \in \N$ such that $\beta^N = 1$.
        \item The map $\varphi$ is \textbf{aperiodic} if for all $n \in \N$, $\beta^n \neq 1$.
    \end{itemize}
\end{defi}

First, we will assume that $\abs{m(0)} = 1$. Indeed, we have the following lemma.

\begin{lem}
    Let $\varphi : \D \to \D$ be elliptic and bijective such that $\varphi(0) = 0$, and $m \in \Hol(\D)$. 
    
    If $W_{m, \varphi}$ is similar to a linear isometry of $\Hol(\D)$, then $\abs{m(0)} = 1$.
\end{lem}

\begin{proof}
    If $\abs{m(0)} < 1$, then considering the map $g = U^{-1} \indic$, with $\indic(z) = 1$, we obtain
    \[ \norm{U^{-1}V^n U g}_{\infty, 0} = \norm{W_{m, \varphi}^n(g)}_{\infty, 0}
        = \abs{m(0)}^n \abs{g(0)} \xrightarrow[n \to \infty]{} 0. \]
    Using the continuity of $U$, we get $1 = \norm{\indic}_{\infty, 0} = \norm{V^n U g}_{\infty, 0} \longrightarrow 0$, which is impossible. \medskip
    
    If $\abs{m(0)} > 1$, then 
    \[ \norm{UV^n U^{-1} \indic}_{\infty, 0} = \norm{W_{m, \varphi}^n(\indic)}_{\infty, 0}
        = \abs{m(0)}^n \xrightarrow[n \to \infty]{} +\infty. \]
    However, there exists $p \in \N$ and $c > 0$ such that
    \[ \norm{UV^n U^{-1} \indic}_{\infty, 0}
        \le c \norm{V^n U^{-1} \indic}_{\infty, 1-1/p}
        = c \norm{U^{-1} \indic}_{\infty, 1-1/p} < +\infty. \]
    This leads to an absurdity.
\end{proof}

Now, we consider separately periodic and aperiodic symbols. \medskip

\uline{2\textsuperscript{nd} case}: Assume that $\varphi : z \mapsto \beta z$ is elliptic periodic. Then, there exists $N \in \N$ such that
\[ W_{m, \varphi}^N = M_{m_N}, \quad \text{with} \quad
    m_N(z) = \prod_{k = 0}^{N-1} (m \circ \varphi^{[k]})(z)
        = \prod_{k=0}^{N-1} m(\beta^k z). \]
        
If $W_{m, \varphi}$ is similar to a linear isometry of $\Hol(\D)$, then so is $W_{m, \varphi}^N = M_{m_N}$.
Using Theorem \ref{Thm - Similarités isométries Hol}.$(ii)$ (proved in Section \ref{Sous-section - Opérateurs de multiplication}), the map $m_N$ must be constant, and unimodular. \medskip

Conversely, if $m_N$ is constant and unimodular, thanks to \cite[Lemma 3.5]{ABBC-rot}, since $\abs{m(0)} = 1$, we may write
\[ m = \exp(\tilde m), \quad \text{with} \quad
    \tilde m = i\theta + \sum_{k=1}^{N-1} z^k f_k(z^N), \qquad
    \theta \in \R, \qquad f_1, \cdots, f_{N-1} \in \Hol(\D). \]
    
Moreover, we show that there exists $U$ invertible such that $UW_{m, \varphi}U^{-1} = V$, with $V$ a linear isometry of $\Hol(\D)$. Indeed, we look at particular $U$ and $V$:
\[ U = M_w, \qquad V = T_{e^{i\theta}, \beta}, \]
with $w \in \Hol(\D)$ not vanishing on $\D$, $e^{i\theta}$ in the definition of $\tilde m$, and $\beta$ in the definition of $\varphi$.
Note that since $w$ does not vanish on $\D$, we may write $w = \exp(\tilde w)$.

Then, for all $h \in \Hol(\D)$ and $z \in \D$,
\begin{align*}
    (UW_{m, \varphi}U^{-1})(h)(z) = V(h)(z)
    & \iff \frac{h(\beta z)}{w(\beta z)} m(z) w(z) 
            = e^{i\theta} h(\beta z) \\
    & \iff \frac{w(z)}{w(\beta z)} \exp\left[\sum_{k=1}^{N-1} z^k f_k(z^N)\right] = 1 \\
    & \iff \exp\left[\tilde w (z) - \tilde w (\beta z) + \sum_{k=1}^{N-1} z^k f_k(z^N)\right] = 1 \\
    & \Longleftarrow \;\; \tilde w (z) - \tilde w (\beta z) + \sum_{k=1}^{N-1} z^k f_k(z^N) = 0.
\end{align*} 
    
Consider $\dsps f_k(z) = \sum_{\ell \ge 0} a_\ell^{(k)} z^\ell$ and $\dsps \tilde w (z) = \sum_{\ell \ge 0} b_\ell z^\ell$.
We are searching for maps $\tilde w$ such that
\begin{equation}\label{Eq - Recherche de w}
   \begin{aligned}
        \tilde w (z) - \tilde w (\beta z) + \sum_{k=1}^{N-1} z^k f_k(z^N)
        & = \sum_{\ell \ge 0} \left[b_\ell (1-\beta^\ell) z^\ell 
            + \sum_{k=1}^{N-1} a_\ell^{(k)} z^{N\ell + k}\right] \\
        & = \sum_{k=1}^{N-1} \left[ \sum_{\ell \ge 0} 
            \left(b_{N\ell + k} (1 - \beta^k) + a_\ell^{(k)}\right) z^{N\ell + k} \right] = 0.
    \end{aligned} 
\end{equation}
Finally, for all $\ell \ge 0$, considering $b_{N\ell} = 0$ and
\[ b_{N\ell + k} = \frac{a_\ell^{(k)}}{\beta^k - 1}, \qquad 1 \le k \le N-1, \]
we obtain the equation (\ref{Eq - Recherche de w}).
Hence, for this particular $w$, we get $M_w W_{m, \varphi} M_w^{-1} = T_{e^{i\theta}, \beta}$. \medskip

\fbox{2\textsuperscript{nd} conclusion:} If $\varphi$ is periodic, then the operator $W_{m, \varphi}$ is similar to a linear isometry of $\Hol(\D)$ if and only if the map $m_N$ is constant and unimodular, that is if and only if $\sigma_p(W_{m, \varphi}) \cap \T \neq \varnothing$, by \cite[Theorem 3.4]{ABBC-rot}. \medskip

\uline{3\textsuperscript{rd} case}: Assume that $\varphi$ is elliptic aperiodic.

If $W_{m, \varphi}$ is similar to a linear isometry of $\Hol(\D)$, then we may write
\[ U^{-1} W_{m, \varphi} U = T_{\alpha, \beta}, \]
for some $\abs{\alpha} = 1$, $\abs{\beta} = 1$. We apply this equality for $f = \indic$. Hence,
\[ (U^{-1}W_{m, \varphi}U)(\indic) 
    = T_{\alpha, \beta}(\indic) = \alpha \indic
    \iff m(U\indic \circ \varphi) = \alpha (U\indic). \]
Therefore, $\alpha \in \T$ is an eigenvalue of $W_{m, \varphi}$. \medskip

Conversely, assume that $\sigma_p(W_{m, \varphi}) \cap \T \neq \varnothing$. Let $\lambda \in \sigma_p(W_{m, \varphi}) \cap \T$, and $f \in \Hol(\D)$ such that
\[ W_{m, \varphi}(f) = m(f \circ \varphi) = \lambda f. \]

Then, for all $z \in \D$, $f(z) \neq 0$. Indeed, assume that there exists $z_0 \in \D$ such that $f(z_0) = 0$. Since $\lambda \neq 0$, by induction, for all $n \in \N_0$, we get $(f \circ \varphi^{[n]})(z_0) = f(\beta^n z_0) = 0$. Because $\varphi$ is aperiodic, it implies that $f = 0$ on $\{z \in \C : \abs{z} = \abs{z_0}\}$, so $f = 0$ on $\D$ using analytic continuation. Let $U = M_f$. Then, $U$ is invertible, and for all $g \in \Hol(\D)$,
\begin{align*}
    (U^{-1} W_{m, \varphi} U)(g) = (U^{-1} W_{m, \varphi})(fg) 
    & = U^{-1}(m(f \circ \varphi)(g \circ \varphi)) \\
    & = \lambda U^{-1} (f(g \circ \varphi)) \\
    & = \lambda (g \circ \varphi) = T_{\lambda, \beta}(g).
\end{align*}
Hence, $W_{m, \varphi}$ is similar to a linear isometry of $\Hol(\D)$.
\medskip

\fbox{Final conclusion:} A weighted composition operator $W_{m, \varphi}$ is similar to a linear isometry of $\Hol(\D)$ if and only if $\varphi$ is bijective and elliptic with $\sigma_p(W_{m, \varphi}) \cap \T \neq \varnothing$. By  \cite[Proposition 3.3 and 3.6]{ABBC-rot} a necessary condition to get a nonempty point spectrum is that $m$ does not vanish on $\D$.  

\section{Isometries of \texorpdfstring{$\Hol(\D)$}{Hol(D)} with a single seminorm}
    \label{Section - Isométries 1 semi-norme}

We now consider the isometries for a \textbf{single} seminorm.
In this case, we expect to get more operators. This is indeed the case.
In order to show it, we will use an analogy with the disc algebra $A(\D)$, a case already discussed by El-Gebeily and Wolfe in \cite{El-G-W}.

We begin this section with a useful proposition.

\begin{prop}[\cite{El-G-W}]\label{Prop - EGW} \;

    Let $S : A(\D) \to A(\D)$ an isometry. Then there exists a closed subset $K \subset \T$, and two maps $\rho : K \to \T$ continuous, $\phi : K \to \T$ continuous and onto, such that for all $f \in A(\D)$ and $\xi \in K$,
    \[ \rho(\xi) \times (Sf)(\xi) = f(\phi(\xi)). \]
\end{prop}

It will allow us to prove the main theorem of this section.

\begin{theo}\label{Thm - 1 semi-norme}
    Let $0 < r < 1$ fixed, and $T : \Hol(\D) \to \Hol(\D)$ a linear continuous operator such that for all $f \in \Hol(\D)$,
    \[ \norm{T(f)}_{\infty, r} = \norm{f}_{\infty, r}. \]
    Then there exist a map $B_1$, which is constant or a finite Blaschke product, and $\beta \in \T$ such that for all $f \in \Hol(\D)$,
    \[ T(f)(z) = B_1(z/r) f(\beta z). \]
\end{theo}

\begin{proof}
    To begin, let us consider $U : \Hol(\D) \to \Hol(r^{-1}\D)$ the operator defined by
    \[ U(f)(z) = f(rz), \qquad \abs{z} \le 1/r. \]
    Then, by setting $S = UTU^{-1} : \Hol(r^{-1}\D) \to \Hol(r^{-1}\D)$, the operator $S$ is isometric for $\norm{\cdot}_{\infty, 1}$, that is, for the disc algebra norm. \medskip

    Note that the set of all polynomials is dense in $\Hol(r^{-1}\D)$ and in $A(\D)$ for the norm $\norm{\cdot}_{\infty, 1}$. Hence, we can extend $S$ to an operator $\tilde S : A(\D) \to A(\D)$ so that $\tilde S$ is still isometric for $\norm{\cdot}_{\infty, 1}$. In the following, we will continue to write $S$ even if we consider $\tilde S$. \medskip

    \uline{Step 1}: Recall that for all $n \in \N_0$, we denote $e_n(z) = z^n$, and $f_n = Se_n$. Let $K$ be the set defined in Proposition \ref{Prop - EGW}. Then for all $\xi \in K$,
    \[ \rho(\xi) \times f_0(\xi) = 1, \qquad 
        \rho(\xi) \times f_1(\xi) = \phi(\xi). \]
        
    Using the first equation, we obtain that for all $\xi \in K$, $\rho(\xi) \neq 0$. We can say more: since $\rho$ is continuous on the compact set $K$, there exists $\delta > 0$ such that $\abs{\rho(\xi)} \le \delta$ for $\xi \in K$. We have the same result for $f_0$. Furthermore,
    \[ \phi(\xi ) = \frac{f_1(\xi)}{f_0(\xi)}, \qquad \xi \in K. \]

    We focus on the Lebesgue measure $m(K)$ of the set $K$ relatively to $\T$. We will show that $m(K) > 0$. Indeed, first, we can say that $K$ is not finite (otherwise, we would have $\phi(K) = \T$ also finite, impossible).

    \begin{itemize}[label=\sbt]
        \item If the map $f_0$ is constant, then $\phi$ is of the form $c f_1$, with $c \in \C$. Hence, we can extend $\phi$ to a holomorphic function on $r^{-1}\D$, which is in particular $\mathcal C^1$ on $\T$.
        Assume that $m(K) = 0$, we would get $m(\phi(K)) = m(\T) = 0$, impossible. \medskip

        \item If the map $f_0$ is not constant, since it is analytic on $r^{-1}\D$, because the zeroes are isolated, there is only a finite number of zeroes of $f_0$ that are unimodular. Denote them by $\xi_1, \cdots, \xi_n$. Therefore, the map $\phi$ can be extended to a $\mathcal C^1$ map on $\T \backslash \{\xi_1, \cdots, \xi_n\}$, such that $\phi(K \backslash \{\xi_1, \cdots, \xi_n\})$ contains $\T$, except at most $n$ points.
        Assume that $m(K) = 0$, we would get $m(\T \backslash \{ n \text{ points}\}) = 0$, impossible.
    \end{itemize}

    Finally, we have proved that $m(K) > 0$. \medskip
    
    \uline{Step 2}: By \cite[Theorem 1.1]{MacDonald} and \cite[Proposition 1]{El-G-W}, there exist a function $\varphi \in H^\infty$ of the form $\varphi = h_1/h_2$, with $h_1, h_2 \in A(\DD)$, and $h_2 = 0$ on a certain subset $L \subset \T$; also, another map $\psi \in A(\D)$ such that $\norm{\varphi}_{\infty, 1} = \norm{\psi}_{\infty, 1} = 1$, $\varphi(K) = \T$, $\psi = 0$ on $L$ and $\abs{\psi} = 1$ on $K$ satisfying
    \[ S(f) = \psi(f \circ \varphi), \qquad f \in A(\D). \medskip \] 

    Recall that at the beginning, $S$ is defined for $f \in \Hol(r^{-1}\D)$. Thus, we obtain $\varphi \in \Hol(r^{-1}\D)$, that is $L = \varnothing$. Hence, the map $\psi$ does not vanish on $\T$. 
    In addition, by setting $X = \{\xi \in \T : \abs{\varphi(\xi)} = 1\}$, if $X$ was finite, we would get $\varphi(X)$ finite. This is impossible since $\varphi(K) = \T$, so $\T \subset \varphi(\overline \D)$. By Theorem \ref{Thm - Jonathan amélioré}, $\varphi$ is a finite Blaschke product. \medskip
    
    Now, assume that $\psi$ is not constant. If $\psi$ is not a finite Blaschke product, note that since $\norm{\psi}_{\infty, 1} = 1$, we have $\psi(\overline \D) \subset \overline \D$. By Theorem \ref{Thm - Jonathan amélioré}, the set $X = \{ \xi \in \T : \abs{\psi(\xi)} = 1\}$ is finite. Choose $z_0 \in \T \backslash X$ such that $\varphi^{-1}(\varphi(z_0)) \subset \T \backslash X$; this is possible since $X$ is finite. Set
    \[ g(z) = \frac{\overline{\varphi(z_0)} z + 1}{2}. \]
    Then $\norm{g}_{\infty, 1} = 1$, and $\abs{g(\varphi(z))} = 1 \iff \varphi(z) = \varphi(z_0) \iff z \in \varphi^{-1}(\varphi(z_0)) \subset \T \backslash X$. Thus,
    \[ \abs{g(\varphi(z))} = 1 \Longrightarrow \abs{\psi(z)} < 1 
        \Longrightarrow \abs{S(g)(z)} = \abs{\psi(z)} \abs{g(\varphi(z))} < 1. \]
    Therefore, $\norm{S(g)}_{\infty, 1} < 1$, which is impossible because $S$ is isometric for $\norm{\cdot}_{\infty, 1}$. To conclude, $X = \T$ (by Theorem \ref{Thm - Jonathan amélioré}), and $\psi$ is a finite Blaschke product. \medskip

    Finally, $S(f)(z) = B_1(z) f(B_2(z))$,
    with $B_1$ and $B_2$ of the wanted form. We have only left to go back to $T$ via the invertible operator $U$, and that gives
    \begin{align*}
        T(f)(z) & = (U^{-1}S)(Uf)(z) \\ 
        & = U^{-1}(B_1(z) \times (Uf)(B_2(z))) \\
        & = U^{-1}(B_1(z) \times f(r B_2(z)) \\
        & = B_1(z/r) f(rB_2(z/r)).
    \end{align*}

    We finish the proof by noting that since the equality must hold for all $f \in \Hol(\D)$, if we denote $\phi(z) = r B_2(z/r)$, we must have $\phi(\D) \subset \D$. 
    Moreover, $\phi(r \overline \D) = r \overline \D$, so by a renormalisation and Lemma \ref{Lemme du train}, $\phi$ is a rotation. Hence, there exists $\abs{\beta} = 1$ such that
    \[ T(f)(z) = B_1(z/r) f(\beta z). \qedhere \]
\end{proof}

\section{Spectral study of the isometries}
    \label{Section - Spectre isométries}

In this section, we make a spectral study of the isometries of $\Hol(\D)$, endowed with two or a single seminorm. We will consider separately the two cases.

\subsection{Spectra of the isometries for two seminorms} \;

First, we focus on the isometries of $\Hol(\D)$, that is the operators
\[ T_{\alpha, \beta} : f \longmapsto \alpha f(\beta \cdot). \]

Following the results of \cite{ABBC-rot}, the results depend on the nature of the number $\beta$ (cf. Definition \ref{Def - Périodique et apériodique}).

\begin{itemize}[label=\sbt]
    \item If $\beta$ is periodic, then $\sigma(T_{\alpha, \beta}) = \sigma_p(T_{\alpha, \beta}) = \{ \alpha \beta^k : k \in \N_0\}$. \smallskip
    
    \item If $\beta$ is aperiodic, then 
    \[ \sigma_p(T_{\alpha, \beta}) = \{ \alpha \beta^k : k \in \N_0\}
        \subset \sigma(T_{\alpha, \beta}) \subset \T, \]
    and if we consider for $\tau > 2$ the set of all \textit{Diophantine numbers of order $\tau$}, defined by
    \[ \mathcal D (\tau) = \{ \xi \in \R : \exists \gamma > 0, 
        \forall p \in \Z, \forall q \in \N, 
        \abs{p/q - \xi} \ge \gamma q^{-\tau} \}, \]
    we obtain a little improvement: for $\beta = e^{2i\pi\xi}$, $\xi \in \mathcal D(\tau)$,
    \[ \sigma(T_{\alpha, \beta}) \subset \{e^{2i\pi x} : x \not\in \Q \} \cup \{1\}. \]
\end{itemize}

\subsection{Spectra of the isometries for one seminorm} \;

Now, let $0 < r < 1$. Using the results of Section \ref{Section - Isométries 1 semi-norme}, we want to find the spectra of the operators defined by
\[ T(f)(z) = B(z/r) f(\beta z), \]
with $B$ constant or a finite Blaschke product, and $\beta \in \T$. \smallskip

\uline{1\textsuperscript{st} case}: If $B$ is a constant map of modulus one, denote by $\xi \in \T$ the constant. Then we obtain $T = T_{\xi, \beta}$, so the spectral study has already been done previously. \medskip

\uline{2\textsuperscript{nd} case}: If $B$ is a finite Blaschke product, then since $B$ is vanishing on $\D$, we obtain $\sigma_p(T) = \varnothing$ (see \cite[Proposition 3.3 and 3.6]{ABBC-rot}). 
Now, to compute the spectrum of $T$, consider once again the two subcases depending on the nature of $\beta$.

\begin{itemize}[label=\sbt]
    \item If $\beta$ is periodic, then there exists $N \in \N$, $N \ge 2$ such that $\beta^N = 1$ and $\beta^k \neq 1$ for $k < N$.
    It follows that for all $f \in \Hol(\D)$,
    \[ T^N (f)(z) = m_N(z) f(z), \qquad 
        m_N(z) = \prod_{j = 0}^{N-1} B\left(\frac{\beta^j z}{r}\right). \]
    Hence, by \cite[Proposition 3.2]{ABBC-rot}, we have $\sigma(T) = \{\lambda \in \C : \lambda^N \in m_N(\D)\}$. \smallskip

    \item If $\beta$ is aperiodic, then by \cite[Corollary 2.2]{ABBC-rot}, $\{\beta^n B(0) : n \in \N_0\} \subset \sigma(T)$. Moreover, using \cite[Proposition 7.2]{ABBC-rot},
    we have $\sigma(T) \subset D(0, M)$, with $M$ defined by
    \[ M = \lim_{R \to 1} M_R, \qquad M_R = \exp\left( \frac{1}{2\pi} \int_0^{2\pi} \log \abs{B(Re^{it}/r)} \dd t \right). \]
    Note that if $m(z) = B(z/r)$, we can write
    \[ m(z) = z^N \tilde m (z), \quad \text{with} \quad
        \tilde m(z) = \frac{e^{i\theta}}{r^N} \prod_{j = 1}^K 
            \frac{z - r\alpha_j}{r - \overline \alpha_j z}, \]
    with $\alpha_1, \cdots, \alpha_K \neq 0$. Since $m$ is well defined in $\overline \D$, the equation (7.2) of \cite{ABBC-rot} is also valid for $M_R$ with $R \to 1$, that is for $M_1 = M$. Hence, 
    \[ M = \abs{\tilde m (0)} \prod_{j = 1}^K \frac{1}{\abs{r \alpha_j}}
        = \abs{\frac{e^{i\theta}}{r^N} \prod_{j = 1}^K \frac{-r\alpha_j}{r}}
            \prod_{j = 1}^K \frac{1}{\abs{r \alpha_j}}
        = \frac{1}{r^N} \prod_{j = 1}^K 
            \frac{\abs{\alpha_j}}{r \abs{\alpha_j}}
        = \frac{1}{r^{N+K}}. \]
    Therefore, if $B$ has $d$ zeroes (counting the multiplicities), then 
    \[ \{\beta^n B(0) : n \in \N_0\} \subset 
            \sigma(T) \subset D(0, r^{-d}). \]
    Thus, the spectrum of $T$ contains a set which is dense in the circle centered at $0$ and of radius $\abs{B(0)}$.
\end{itemize}

%%%%%%%%%%%%%%%%%%%%%%%%%%%%%%%%%%%%%%%%%%%%%%%%%%%%%%%%%%%%%%%%%%%%%%%%%%
\bigskip

\noindent \textbf{Acknowledgments:}
This research is partly supported by the Bézout Labex, funded by ANR, reference ANR-10-LABX-58.
The authors are grateful to an anonymous referee for comments.  The first two authors also wish to thank E. Abakoumov for stimulating questions and suggestions,
while the third author thanks
Michael Pilla for drawing our attention to the results of \cite{Kamowitz}.

\end{document}